\documentclass{daj}

\dajAUTHORdetails{%
  title = {Slice Rank and Analytic Rank for Trilinear Forms}, 
  author = {Amichai Lampert},
  plaintextauthor = {Amichai Lampert},
}   

\dajEDITORdetails{%
   year={2025},
   number={25},
   received={13 February 2025},   
   published={9 October 2025},  
   doi={10.19086/da.145198},       
}   

\usepackage{amscd,amssymb,amsmath,amsfonts,amsthm}
\usepackage{comment,hyperref}
\usepackage{xcolor}

\newtheorem{theorem}{Theorem}[section]
\newtheorem{lemma}[theorem]{Lemma}

\newtheorem{definition}[theorem]{Definition}

\newtheorem{question}[theorem]{Question}
\newtheorem{conjecture}[theorem]{Conjecture}

\renewcommand{\k}{{\mathbf{k}}}

\newcommand{\E}{\mathop{\mathbb{E}}}
\newcommand{\F}{\mathbb{F}}

\renewcommand{\P}{\mathop{\mathbb{P}}}

\newcommand{\codim}{\textnormal{codim}}
\newcommand{\rk}{\textnormal{rk}}
\newcommand{\ark}{\textnormal{ark}}
\newcommand{\prk}{\textnormal{prk}}
\newcommand{\srk}{\textnormal{srk}}

\begin{document}

\begin{frontmatter}[classification=text]


\author[AL]{Amichai Lampert\thanks{This work was supported by NSF grant DMS-1840234}}

\begin{abstract}
In this note, we present an elementary proof of the fact that the slice rank of a trilinear form over a finite field is bounded above by a linear expression in the analytic rank. The existing proofs by Adiprasito-Kazhdan-Ziegler and Cohen-Moshkovitz both rely on results of Derksen via geometric invariant theory. A novel feature of our proof is that the linear forms appearing in the slice rank decomposition are obtained from the trilinear form by fixing coordinates.
\end{abstract}
\end{frontmatter}

\section{The theorem}
Let $\k = \F_q$ be a finite field and let $U,V,W$ be finite dimensional vector spaces over $\k.$  
\begin{definition}\label{rk-def} 
For a trilinear form $f:U\times V\times W\to\k$ we are interested in two kinds of rank.
\begin{itemize}
    \item The \emph{slice rank} of $f$ is 
    \[
    \srk(f) = \min\left\{r: f = \sum_{i=1}^r \alpha_i\cdot h_i\right\},
    \]
    where $\alpha_i,h_i$ are linear and bilinear, respectively, in disjoint sets of variables.
    \item The \emph{analytic rank} of $f$ is 
    \[
    \ark(f) = -\log_q \frac{|Z|}{|U\times V|},
    \]
    where $Z = \{(u,v)\in U\times V: f(u,v,\cdot) \equiv 0\}.$
\end{itemize}
    
\end{definition}

The definition of slice rank goes back to the work of Schmidt \cite{S} on systems of polynomial equations. It was reidiscovered in \cite{T} and used to give a reformulation of Ellenberg and Gijswijt's work on the capset problem. Analytic rank was introduced in \cite{GW}. The inequality $\ark(f)\le \srk(f)$ is straightforward, see \cite{KZ} or \cite{L}. In the other direction, Adiprasito-Kazhdan-Ziegler \cite{AKZ} and Cohen-Moshkovitz \cite{CM-cubics} independently proved that $\srk(f) \ll \ark(f).$ Their proofs use results of Derksen \cite{D} which rely on the powerful tools of geometric invariant theory. The goal of this note is to give an elementary proof.

\begin{theorem}\label{main} For any finite field $\k =\F_q$ and trilinear form $f,$ we have
$$\srk(f) \le 5\cdot \ark(f)+4\cdot \log_q(\ark(f)+1)+29.$$
In addition, the linear forms in the corresponding slice rank decomposition are obtained from $f$ by fixing coordinates.
\end{theorem}

Theorem \ref{main} is the first nontrivial case of a more general conjecture relating analytic rank to another kind of rank, \emph{partition rank}. To state the general conjecture, we need some definitions. Let $V_1,\ldots,V_d$ be finite-dimensional vector spaces over $\k$ and let $f:V_1\times\ldots\times V_d\to \k$ be a multilinear form.

\begin{definition}
    \begin{itemize}
    \item The \emph{partition rank} of $f$ is 
    \[
    \prk(f) = \min\left\{r: f = \sum_{i=1}^r g_i\cdot h_i\right\},
    \]
    where $g_i,h_i$ are multilinear in disjoint sets of variables. Note that if $d=3$ this agrees with the slice rank.
    \item The \emph{analytic rank} of $f$ is 
    \[
    \ark(f) = -\log_q \frac{|Z|}{|V_1\times\ldots\times V_{d-1}|},
    \]
    where $Z = \{(x_1,\ldots,x_{d-1})\in V_1\times\ldots\times V_{d-1}: f(x_1,\ldots,x_{d-1},\cdot) \equiv 0\}.$
\end{itemize}
    
\end{definition}

Again, the inequality $\ark(f)\le \prk(f)$ is not difficult, see \cite{KZ} and \cite{L}.  Lovett \cite{L} and Adiprasito-Kazhdan-Ziegler \cite{AKZ} conjectured that the reverse inequality also holds, up to a constant factor.

\begin{conjecture}
    There exists a constant $C_d$ such that for any finite field and multilinear form $f$ we have 
    \[
    \prk(f) \le C_d\cdot \ark(f).
    \]
\end{conjecture}

The current best bound in this direction is 
$$\prk(f) \le C_d\cdot \ark(f) \cdot (\log_q (1+\ark(f))+1),$$
due to Moshkovitz-Zhu \cite{MZ}.

\section{Proof of theorem \ref{main}}

Let $r=\ark(f).$ For $u\in U$ we write $f[u]$ for the bilinear form $f[u]:V\times W\to \k $ obtained by fixing the $U$ coordinate. For $v\in V$ we write $f\langle v\rangle$ for the bilinear form $f\langle v\rangle: U\times W\to \k$ obtained by fixing the $V$ coordinate. For a finite set $X,$ denote $\E_{x\in X} = \frac{1}{|X|} \sum_{x\in X}.$ Our first two lemmas are inspired by a recent observation of Moshkovitz-Zhu \cite{MZ}.
\begin{lemma}
    \[
    \E_{(u,v)\in Z} q^{\rk f[u]} = \E_{(u,v)\in Z} q^{\rk f\langle v\rangle} = q^r.
    \]
\end{lemma}

\begin{proof}
    A straightforward computation. Let
    $$p(u) = \frac{|\{v\in V:(u,v)\in Z\}|}{|Z|} = q^{-\rk f[u]} \frac{|V|}{|Z|}$$
    be the marginal probability mass function for $u\in U.$ Then 
    \begin{align*}
        \E_{(u,v)\in Z} q^{\rk f[u]} &=  \sum_u p(u)\cdot  q^{\rk f[u]} \\
        &= \sum_u \frac{|V|}{|Z|} = \frac{|U|\cdot|V|}{|Z|} = q^r.  
    \end{align*}
    The proof for $\rk f\langle v\rangle$ is identical.
\end{proof}

This allows us to find a large subspace $U'\subset U$ where $\rk f[u]$ is almost surely small.

\begin{lemma}\label{subspace}
    There exists a subspace $U'\subset U$ with $\codim U' \le r+1$ and 
    \[
    \P_{u\in U'} (\rk f[u] > r+s) < q^{1-s}
    \]
    for every $s>0.$ The linear forms defining $U'$ are obtained from $f$ by fixing coordinates.
\end{lemma}

\begin{proof}
     By linearity of expectation,
    \[
    \E_{(u,v)\in Z} (q^{\rk f[u]}+q^{\rk f\langle v\rangle}) = 2q^r \le q^{r+1}.
    \]
    Therefore, there must exist some $v_0\in V$ with
    \[
    q^{\rk f\langle v_0\rangle}+\E_{u\in Z(v_0)}  q^{\rk f[u]} \le q^{r+1},
    \]
    where $Z(v_0) = \{u\in U: (u,v_0)\in Z\}.$ Choosing $U' = Z(v_0),$ we have \linebreak $\codim U' = \rk f\langle v_0\rangle \le r+1$ and 
    \[
    \P_{u\in U'} (\rk f[u] > r+s) = \P_{u\in U'} (q^{\rk f[u]} > q^{r+s}) < q^{1-s} 
    \]
    by Markov's inequality.
\end{proof}

The final ingredient is a lemma of Shpilka-Haramaty \cite{HS} regarding subspaces of bilinear forms of bounded rank. We include the proof for the reader's convenience.

\begin{lemma}\label{low-rk-der}
    If $g:U\times V\times W\to k$ is a trilinear form with
    $$\P_{u\in U} (\rk g[u] > t) < \frac{q-1}{2qt}$$ for some positive integer $t$ then $\srk (g) \le 4t.$ Moreover, the linear forms in the slice rank decomposition are obtained by fixing coordinates of $g.$
    \end{lemma}

    \begin{proof}
         Let $A = \{u\in U: \rk g[u] \le t \}.$ Our assumption is that $\frac{|A|}{|U|} > 1-\frac{q-1}{2qt}.$ The proof proceeds by induction on $t.$
         
         \textbf{Base case $t=1:$} Suppose first that $g[u] = 0$ for all $u\in A.$ Since $|A| > |U|/q$ and $u\mapsto g[u]$ is a linear map, we deduce that $g[u] = 0$ for all $u\in U,$ so $g = 0.$  We may henceforth assume that there is some $u_0\in A$ with $\rk g[u_0] = 1$. Writing $g[u_0] = \alpha(v)\beta(w),$ we claim that for all $u\in A\cap (A-u_0)$ the bilinear form $g[u]$ is contained in the ideal $(\alpha,\beta).$  Indeed, if $g[u] = \gamma(v)\delta(w)$ with $\gamma \not\in \text{span}(\alpha)$ and $\delta\not\in\text{span}(\beta)$ then 
         \[
         \rk g[u_0+u] = \rk(\alpha(v)\beta(w)+\gamma(v)\delta(w)) = 2,
         \]
         contradicting the fact that $u_0+u \in A.$ Therefore, if $V_0 = \{v: \alpha(v) = 0\}$ and $W_0 = \{w: \beta(w) = 0\}$ and $\tilde g = g\restriction_{U\times V_0\times W_0},$ we get that $\tilde g[u] = 0$ whenever $u\in A\cap (A-u_0),$  a set of density greater than $ 1- 2\cdot \frac{q-1}{2q} = \frac{1}{q}.$ By the linearity of $u\mapsto \tilde g[u]$ we deduce $\tilde g = 0$, so $g\in(\alpha,\beta)\implies \srk(g) \le 2.$

         \textbf{Inductive step:} We may assume that there is some $u_0\in A$ with $\rk g[u_0] = t,$ otherwise we're done by the inductive hypothesis. Write $g[u_0] = \sum_{i=1}^t \alpha_i(v)\beta_i(w)$ and set $V_0 = \{v:\alpha_i(v) = 0\ \forall i\in[t]\},$ $W_0 = \{w:\beta_i(w) = 0\ \forall i\in[t]\}.$ We claim that $\tilde g = g\restriction_{U\times V_0\times W_0}$ satisfies $\rk \tilde g[u] \le t/2 $ for all $u\in A\cap (A-u_0).$ This will complete the proof because this set of $u$'s has density greater than $ 1-2\cdot \frac{q-1}{2qt} \ge 1- \frac{q-1}{2q\lfloor t/2\rfloor}, $ so the inductive hypothesis implies 
         \[
         \srk (\tilde g) \le 4\lfloor t/2\rfloor \le 2t \implies \srk(g) \le 4t.
         \]
        To prove the claim about $\rk \tilde g[u],$ suppose it has rank $s$ and write $\tilde g[u] = \sum_{i=1}^s \tilde\gamma_i(v) \tilde \delta_i(w).$ Let $\gamma_i,\delta_i$ be linear forms on $V,W$ which restrict to $\tilde\gamma_i,\tilde\delta_i$, respectively.  Note that our assumptions on $\rk g[u_0]$ and $\rk \tilde g[u]$ imply that  both $\alpha_i,\gamma_i$ and  $\beta_i,\delta_i$ are linearly independent. Choosing bases, we may identify $V=\k^n,W=\k^m$ and 
        $$ \gamma_i(v) = v_i,\ \delta_i(w) = w_i,\  \alpha_i(v) = v_{s+i},\ \beta_i(w) = w_{s+i}.$$ This yields an equation
         \[
         g[u] = \sum_{i=1}^s v_iw_i +\sum_{j=1}^t  v_{s+j} \phi_j (w)+\sum_{j=1}^t  w_{s+j} \tau_j (v).
         \]
         Decomposing $$\phi_j(w) = \sum_{i=1}^s a_{i,j} w_i + \phi'(w_{s+1},\ldots,w_m),\ \tau_j(v) = \sum_{i=1}^s b_{i,j} v_i + \tau'(v_{s+1},\ldots,v_n),$$ 
         we get 
         \[
         g[u] = \sum_{i=1}^s  \gamma'_i(v) \delta'_i(w) +q(v_{s+1},\ldots,v_n,w_{s+1},\ldots,w_m),
         \]
         where 
         \[
         \gamma'_i(v) = v_i+\sum_{j=1}^t a_{i,j} v_{s+j},\ \delta'_i(w)= w_i+\sum_{j=1}^t b_{i,j} w_{s+j}.
         \]
         The collection of linear forms $\gamma'_1,\ldots,\gamma'_s,v_{s+1},\ldots,v_n$ spans $V^*$ and so is linearly independent, likewise for $\delta'_1,\ldots,\delta'_s,w_{s+1},\ldots,w_m.$ This means that 
         \begin{equation}\label{rk-q}
             t\ge \rk g[u] = s+ \rk(q).
         \end{equation}
         Since $u+u_0\in A,$ we get 
         \begin{align*}
             t &\ge \rk(g[u]+g[u_0]) = \rk \left(\sum_{i=1}^s  \gamma'_i(v) \delta'_i(w)+\sum_{i=1}^t v_{s+i} w_{s+i}+q\right) \\
             &= s+\rk\left(\sum_{i=1}^t v_{s+i} w_{s+i}+q\right) \ge s+t-\rk(q) \ge 2s
         \end{align*}
         where the last inequality used \eqref{rk-q}. This proves that $\rk\tilde g[u] = s\le t/2,$ completing the proof of the lemma.
    \end{proof}
    Now we are ready to put everything together.

    \begin{proof}[Proof of theorem \ref{main}]
    Let $U'$ be the subspace of lemma \ref{subspace} and let \linebreak $g = f\restriction_{U'\times V\times W}.$ For $s = \lceil\log_q (r+1)\rceil+6,$ we get 
    \begin{align*}
        \P_{u\in U'} (\rk f[u] > r+s) &< q^{1-s} \le \frac{1}{q^5(r+1)}  \\
        &< \frac{1}{4(r+s)} \le \frac{q-1}{2q}\cdot \frac{1}{r+s},
    \end{align*}
    where the inequality at the start of the second line follows from 
    \[
    q^5(r+1) > 4r+8r+28 > 4(r+\lceil\log_q (r+1)\rceil+6).
    \]
    By lemma \ref{low-rk-der}, we deduce that $\srk(g) \le 4(r+s) $ and so
    \[
    \srk(f) \le r+1+4(r+s) \le 5r+4\log_q(r+1)+29.
    \]
    \end{proof}

    The proof of theorem \ref{main} is asymmetric in its treatment of $U,V,W.$ While only $r+1$ $U$-linear forms are needed, we require $\approx 2r$ linear forms on each of $V,W.$ One might hope that a more symmetric treatment would require only $r+1$ linear forms in each of $U,V,W.$

    \begin{question}
        Is there an elementary proof that $\srk(f) \le 3\cdot \ark(f)+o(\ark(f))$? 
    \end{question}





\section*{Acknowledgments} 
The author would like to thank Daniel Altman, Guy Moshkovitz and Tamar Ziegler for helpful discussions. He is also grateful to the anonymous reviewers for their careful reading and useful suggestions.

\bibliographystyle{amsplain}

\begin{thebibliography}{10}

\bibitem{AKZ}
Karim {Adiprasito}, David {Kazhdan}, and Tamar {Ziegler}.
\newblock {On the Schmidt and analytic ranks for trilinear forms}.
\newblock {\em arXiv e-prints}, page arXiv:2102.03659, February 2021.

\bibitem{CM-cubics}
Alex Cohen and Guy Moshkovitz.
\newblock Structure vs. randomness for bilinear maps.
\newblock {\em Discrete Anal.}, pages Paper No. 12, 21, 2022.

\bibitem{D}
Harm Derksen.
\newblock The {$G$}-stable rank for tensors and the cap set problem.
\newblock {\em Algebra Number Theory}, 16(5):1071--1097, 2022.

\bibitem{GW}
W.~T. Gowers and J.~Wolf.
\newblock Linear forms and higher-degree uniformity for functions on {$\mathbb{F}^n_p$}.
\newblock {\em Geom. Funct. Anal.}, 21(1):36--69, 2011.

\bibitem{HS}
Elad Haramaty and Amir Shpilka.
\newblock On the structure of cubic and quartic polynomials.
\newblock In {\em S{TOC}'10---{P}roceedings of the 2010 {ACM} {I}nternational {S}ymposium on {T}heory of {C}omputing}, pages 331--340. ACM, New York, 2010.

\bibitem{KZ}
David Kazhdan and Tamar Ziegler.
\newblock Approximate cohomology.
\newblock {\em Selecta Math. (N.S.)}, 24(1):499--509, 2018.

\bibitem{L}
Shachar Lovett.
\newblock The analytic rank of tensors and its applications.
\newblock {\em Discrete Anal.}, pages Paper No. 7, 10, 2019.

\bibitem{MZ}
Guy {Moshkovitz} and Daniel~G. {Zhu}.
\newblock {Quasi-linear relation between partition and analytic rank}.
\newblock {\em arXiv e-prints}, page arXiv:2211.05780, November 2022.

\bibitem{S}
Wolfgang~M. Schmidt.
\newblock The density of integer points on homogeneous varieties.
\newblock {\em Acta Math.}, 154(3-4):243--296, 1985.

\bibitem{T}
Terence Tao.
\newblock A symmetric formulation of the croot-lev-pach-ellenberg-gijswijt capset bound.
\newblock \url{https://terrytao.wordpress.com/2016/05/18/a-symmetric-formulation-of-the-croot-lev-pach-ellenberg-gijswijt-capset-bound/}.

\end{thebibliography}


\begin{dajauthors}
\begin{authorinfo}[AL]
  Amichai Lampert\\
  NSF Research Fellow and Postdoctoral Assistant Professor \\
  University of Michigan\\
  amichai@umich.edu \\
\url{https://sites.google.com/umich.edu/amichai/}
\end{authorinfo}
\end{dajauthors}

\end{document}